\documentclass[11pt]{amsart}
\usepackage{amssymb}
\newcommand{\no}[1]{#1}

\renewcommand{\no}[1]{}

\no{\usepackage{times}\usepackage[
subscriptcorrection, slantedGreek, 
nofontinfo]{mtpro}
\renewcommand{\Delta}{\upDelta}
}

\date{\today}

\newtheorem{theorem}{Theorem}
\newtheorem{proposition}{Proposition}

\newtheorem{corollary}{Corollary}

\theoremstyle{remark}
\newtheorem{remark}{Remark}

\DeclareMathOperator{\supp}{supp}

\DeclareMathOperator{\dist}{dist}

\newcommand{\eps}{\varepsilon}

\newcommand{\R}{{\bf R}}
\newcommand{\Id}{\mbox{Id}}
\renewcommand{\r}[1]{(\ref{#1})}
\newcommand{\PDO}{$\Psi$DO}
\newcommand{\be}[1]{\begin{equation}\label{#1}}
\newcommand{\ee}{\end{equation}}

\renewcommand{\d}{\mathrm{d}}
\renewcommand{\i}{\mathrm{i}}
\newcommand{\bo}{\partial \Omega}

\title[Instability of the linearized problem in multiwave tomography]{Instability of the linearized problem in  multiwave tomography  of  recovery both the source and the speed}

\author[P. Stefanov]{Plamen Stefanov}
\address{Department of Mathematics, Purdue University, West Lafayette, IN 47907}
\thanks{First author partly supported by a NSF  Grant DMS-0800428}

\author[G. Uhlmann]{Gunther Uhlmann}
\address{Department of Mathematics, University of Washington, Seattle, WA 98195}
\thanks{Second author partly supported by NSF and  a Walker Family Endowed Professorship}

\begin{document}

\begin{abstract}
In this paper we consider the linearized problem of recovering both the sound speed and the thermal absorption arising in thermoacoustic and photoacoustic tomography. We show that the problem is unstable in any scale of Sobolev spaces.
\end{abstract}

\maketitle

\section{Introduction}  
In multiwave tomography, one sends one type of wave to the body of a patient, most often electromagnetic (thermoacoustic tomography) or optical radiation (photoacoustic tomography) which interacts with the tissue, and measures the acoustic signal on the boundary generated by this interaction. This combines the high contrast of the incoming waves with the high resolution of the measured ultrasound ones. The mathematical model of the emitted ultrasound wave is the following. 
Let $u$ solve the problem
\begin{equation}   \label{1}
\left\{
\begin{array}{rcll}
(\partial_t^2 -c^2\Delta)u &=&0 &  \mbox{in $(0,T)\times \R^n$},\\
u|_{t=0} &=& f,\\ \quad \partial_t u|_{t=0}& =&0, 
\end{array}
\right.               
\end{equation}
where the sound speed $c=c(x)>0$ and $T>0$ are fixed. 
Assume that $f$ and $c-1$ are supported in $\bar\Omega$, where $\Omega\subset \R^n$ is some  bounded domain with a smooth convex boundary. The measurements are modeled by the operator
\be{1b}
\Lambda_1 f : = u|_{[0,T]\times\partial\Omega}.
\ee
The first step  in multiwave imaging is to recover $f$ given $\Lambda_1f$. The speed $c$ is usually assumed to be known but in practice, it is not. Then the natural question is whether we can recover both $c$ and $f$. The answer is still unknown. A discussion of this problem, together with a partial local result stating that $c$ can be determined up to a constant scaling  can be found in \cite{HristovaKu08}. David Finch observed a link between this problem and the transmission eigenvalues. 

The problem of recovery of $f$, given $c$ has received a lot of attention in the past years. We refer to \cite{AgrKuchKun2008,Gaik_2010, FinchHR_07, finchPR, Finch_Rakesh_2006,FinchRakesh08, Haltmeier04, Haltmeier05, Hristova08,HristovaKu08, Kruger03,Kruger99,KuchmentKun08, Patch04,QSUZ_skull, SU-thermo, SU-thermo_brain,XuWang06,XuKA} for some works in this direction. If $c$ is constant, then the inversion of $\Lambda_1$ is actually an integral geometry problem as well. For variable $c$, there is always uniqueness if $T\gg1$, and there is stability if $c$ is non-trapping and $T\gg1$. We refer to \cite{SU-thermo, S-U-InsideOut10} for details. 

When an inaccurate speed is used for reconstruction by time reversal, the resulting images are distorted, and full of ``artifacts'', see for example the images in \cite{Jose2012} or \cite{HristovaKu08}. The mathematical structure of the artifacts is easy to understand: the forward map $\Lambda_1$ is a Fourier Integral Operator (FIO) with a canonical relation given by the graph of the map determined by the geodesics rays from the interior to the boundary, see \cite{SU-thermo} for more details. The time reversal is another FIO with a canonical relation given by the graph of the inverse of the former one. When the speed is different, we get that time reversal with a wrong speed is an FIO with a canonical relation not the diagonal (the graph of the identity) but the graph of the composition of the forward one and a backward one with a different speed. In particular, since the canonical relation of $\Lambda_1$ and its reversal have two disconnected components, one can see double images of the same singularity, well visible in the examples in \cite{Jose2012}, for instance. An algorithm to tune in the speed by maximizing  the sharpness of the reconstructed $f$ is proposed in \cite{Treeby1011}. This is related to the FIO description of the artifacts but  good mathematical understanding of this algorithm is lacking. 

One of the ways to recover the speed is to take additional measurements and to recover $c$ from travel times. This is the method proposed in \cite{XW} in thermoacoustic tomography. The travel time problem is stable under some geometric assumptions on $c$  ($c^{-2}\d x^2$ being a simple metric in the domain is enough) which are satisfied when $c$ is close enough to a constant, in particular, see, e.g., \cite{Sh-book} and the references there. Then $f$ can be recovered stably as well.
Additional data can be provided either by placing ultrasound sources around the body, or by placing passive absorbing (tissue imitating) objects around the body which become ultrasound sources by the thermo- or the photo-acoustic effect. We refer to the recent paper \cite{Jose2012} for references and numerical and experimental implementations of that method. As can be expected, the results are very good. 
Simultaneous reconstruction of $f$ and $c$ aside from the above mentioned work \cite{Treeby1011}, have been tried  with various success  in \cite{Yuan2006,Yuan:09,Zhang_2006}, for example. 

In this paper, we study the linearization $\delta\Lambda_1$  and we show that the latter is unstable. In particular, we prove the following. 

\begin{theorem}\label{thm_lin}
There is no stability estimate of the type
\[
\|\delta f\|_{H^{s_1}(\Omega)} + \|\delta c^2\|_{H^{s_1}(\mathcal{K})}\le C\big\|\delta\Lambda_1  \{\delta f,\delta c^2\} \big\|_{H^{s_2}},
\]
$s_1\ge0$, $s_2\ge0$, regardless of $s_1$, $s_2$. 
\end{theorem}
We also show that a conditional type of stability estimate cannot hold either, see Remark~\ref{rem1}. 
This suggests instability of the non-linear problem as well but does not imply it directly. Stability of the linearization in some Sobolev norms, even if    not in the sharp ones, does imply (conditional H\"older)    stability of the non-linear problem \cite{SU-JFA09}. The converse however is a much more delicate question. 
To prove the instability of the linearization, we show that the latter is a smoothing operator for $(\delta f,\delta c^2)$ belonging to an explicitly defined infinite dimensional linear space, see \r{V1}. We refer to section~\ref{sec_4} for more details. 

\section{Preliminaries}
Notice first that $c^2\Delta$ is formally self-adjoint w.r.t.\ the measure $c^{-2}\d x$.   
Given a domain $U$, and a function $u(t,x)$, define the energy
\[
E_U(u,t) = \int_U\left( |\nabla_x u|^2   +c^{-2}|u_t|^2 \right)\d x.
\]
In particular, we define the space $H_{D}(U)$ to be the completion of $C_0^\infty(U)$ under the Dirichlet norm
\be{2.0H}
\|f\|_{H_{D}}^2= \int_U |Du|^2 \,\d x.
\ee
It is easy to see that $H_{D}(U)\subset H^1(U)$, if $U$ is bounded with smooth boundary, therefore, $H_{D}(U)$ is topologically equivalent to $H_0^1(U)$. If $U=\R^n$, this is true for $n\ge3$ only.  By the finite speed of propagation, the solution with compactly supported Cauchy data always stays in $H^1$ even when $n=2$.  
The energy norm for the Cauchy data $[f_1,f_2]$, that we denote by $\|\cdot\|_{\mathcal{H}}$ is then defined by
\[
\|[f,f_2]\|^2_{\mathcal{H}} = \int_U\left( |\nabla_x f_1|^2    +c^{-2}|f_2|^2 \right)\d x.
\]
This defines the energy space 
\[
\mathcal{H}(U) = H_D(U)\oplus L^2(U).
\] 
Here and below, $L^2(U) = L^2(U; \; c^{-2}\d x)$. Note also that 
\be{Pf}
\|f\|^2_{H_D} = (-c^2\Delta f,f)_{L^2}.
\ee
The wave equation then can be written down as the system
\be{s1}
\mathbf{u}_t= \mathbf{P}\mathbf{u}, \quad \mathbf{P} = \begin{pmatrix} 0&I\\P&0 \end{pmatrix},\quad P := c^2\Delta,
\ee
where $\mathbf{u}=[u,u_t]$ belongs to the energy space $\mathcal{H}$. The operator $\mathbf{P}$ then extends naturally to a skew-selfadjoint operator on $\mathcal{H}$ if $c\in L^\infty$, and $c^{-1}\in L^\infty$. In this paper, we will deal with either $U=\R^n$ or $U=\Omega$. In the latter case, the definition of $H_D(U)$ reflects Dirichlet boundary conditions.

We generalize next the results in \cite{SU-thermo} to the inverse problem with general Cauchy data $(f_1,f_2)$ in \r{1} with $g$ not necessarily zero. What we really need later is Proposition~\ref{pr_12} only. 
Let $u$ solve the problem
\begin{equation}   \label{1'}
\left\{
\begin{array}{rcll}
(\partial_t^2 -c^2\Delta)u &=&0 &  \mbox{in $(0,T)\times \R^n$},\\
u|_{t=0} &=& f_1,\\ \quad \partial_t u|_{t=0}& =&f_2, 
\end{array}
\right.               
\end{equation}
where $T>0$ is fixed. Set $\mathbf{f}=[f_1,f_2]$. Then for $\mathbf{f}\in\mathcal{H}$, we have $\mathbf{u}\in C([0,T]; \; \mathcal{H})$. 

Assume that $f$ is supported in $\bar\Omega$, where $\Omega\subset \R^n$ is some smooth bounded domain. Set
\be{l1}
{\Lambda} \mathbf{f} : = u|_{[0,T]\times\partial\Omega}.
\ee
The trace $\Lambda \mathbf{f}$ is well defined in $C_{(0)}\big([0,T]; \; H^{1/2}(\bo)\big)$, where the subscript $(0)$ indicates that we take  the subspace of functions $h$ so that $h=0$ for $t=0$. For a discussion of other mapping properties, we refer to \cite{Isakov-book}. When $\mathbf{f}$ is restricted to functions in $\mathcal{H}$, supported in a fixed compact $\mathcal{K}\subset\Omega$, then $\Lambda$ is an FIO with a canonical relation of graph type, and maps $\mathbf{f}$ continuously into $H_{(0)}^1\left([0,T]\times\bo\right)$, see \cite{SU-thermo}. 

Given $h$, let $v$ solve
\begin{equation}   \label{l2}
\left\{
\begin{array}{rcll}
(\partial_t^2 -c^2\Delta)v &=&0 &  \mbox{in $(0,T)\times \Omega$},\\
v|_{[0,T]\times\partial\Omega}&= &h,\\
v|_{t=T} &=& \phi,\\ \quad   \partial_t v|_{t=T}& =&0, \\
\end{array}
\right.               
\end{equation}
where $\phi$ solves the elliptic boundary value problem
\be{l3}
\Delta\phi=0, \quad 
\phi|_{\partial\Omega} = h(T,\cdot).
\ee 
Then we define the following pseudo-inverse
\be{l4}
\mathbf{A} h := [v(0,\cdot),v_t(0,\cdot)] \quad \mbox{in $\bar\Omega$}.
\ee
By  \cite{LasieckaLT},
\[
\mathbf{A} : H^1_{(0)}([0,T]\times \bo) \to \mathcal{H}  \cong H_0^1(\Omega)\times L^2(\Omega)
\]
is a continuous map. Note that the mapping properties above allow us to apply $\mathbf{A}$ to $\Lambda\mathbf{f}$ only when $\mathbf{f}$ is compactly supported in $\Omega$ but the theorem above shows that $\mathbf{A}\Lambda$ extends continuously to the whole $\mathcal{H}$. 

Let $T(\Omega)$ be the length of the longest geodesic in $\bar\Omega$, when $(\Omega, c^{-2}dx^2)$ is non-trapping. 

\begin{theorem}  \label{thm2.1} Let $(\Omega, c^{-2}dx^2)$ be non-trapping, and let $T>T(\Omega)$. Then 
$\mathbf{A}\Lambda=\textbf{Id}-\mathbf{K}$, where $\mathbf{K}$ is compact in $\mathcal{H}(\Omega)$, and   $\|\mathbf{K}\|_{\mathcal{H}(\Omega)}<1$. 
In particular, $\textbf{Id}-\mathbf{K}$ is invertible on $\mathbf{H}(\Omega)$, and $\Lambda$  has an explicit left inverse of the form
\be{2.2}
\mathbf{f} = \sum_{m=0}^\infty \mathbf{K}^m \mathbf{A} h, \quad h:= \Lambda\mathbf{ f}.
\ee
\end{theorem}
\begin{proof}
Let $\mathbf{f}\in C_0^\infty(\Omega)\times C_0^\infty(\Omega)$ first. Let $w$ solve
\begin{equation}   \label{2.3}
\left\{
\begin{array}{rcll}
(\partial_t^2 -c^2\Delta)w &=&0 &  \mbox{in $(0,T)\times \Omega$},\\
w|_{[0,T]\times\partial\Omega}&= &0,\\
w|_{t=T} &=& u|_{t=T}-\phi,\\ \quad w_t|_{t=T}& =&u_t|_{t=T},\\
\end{array}
\right.               
\end{equation}
where $u$ solves \r{1'} with a given $\mathbf{f}\in \mathcal{H}$. 
Let $v$ be the solution of \r{l2} with $h=\Lambda f$. Then $v+w$ solves the same initial boundary value problem in $[0,T]\times\Omega$ that $u$ does (with initial conditions at $t=T$), therefore $u=v+w$. Restrict this to $t=0$ to get
\[
\mathbf{f}= \mathbf{A}\Lambda  \mathbf{f} + \mathbf{w}(0,\cdot).
\]
Set 
\[
\mathbf{K} \mathbf{f} = \mathbf{w}(0,\cdot) = [w(0,\cdot), w_t(0,\cdot)].
\]
We will show now that $\mathbf{K}$ extends to a compact operator. Since $T>T(\Omega)$, all singularities starting from $\bar \Omega$ leave $\bar \Omega$ at $t=T$. Therefore, $u(T,\cdot)$ and $u_t(T,\cdot)$, restricted to $\bar\Omega$, are $C^\infty$. Moreover, considered as linear operators of $\mathbf{f}$, they are operators with smooth Schwartz kernels. 
Then so is $\phi$, see \r{l3}, by elliptic regularity. 
Therefore, the map 
\[
\mathcal{H}(\Omega)\ni \mathbf{f}\quad \longmapsto\quad \left[u(T,\cdot)-\phi ,u_t(T,\cdot)\right]\in \mathcal{H}(\Omega),
\]
defined a priori on $C_0^\infty(\Omega)\times C_0^\infty(\Omega)$ functions, extends to a compact one  because it is an operator with smooth kernel on $\bar\Omega$.   Since the solution operator of \r{2.3} from $t=T$ to $t=0$ is unitary in $ \mathcal{H}(\Omega)$, we get that the map $\mathcal{H}(\Omega)\ni \mathbf{f}\mapsto [w(0,\cdot),w_t(0,\cdot)] \in H_D(\Omega))$ is compact, too, as a composition of a compact and a bounded one. 

We know remove the smoothness restriction on $\mathbf{f}$, and let it be any element in $\mathcal{H}$. 
In what follows, $(\cdot,\cdot)_{H_{D}(\Omega)}$ is the inner product in $H_{D}(\Omega)$, see \r{2.0H}, applied to functions that belong to $H^1(\Omega)$ but maybe not to $H_{D}(\Omega)$ (because they may not vanish on $\bo$). Set $u^T := u(T,\cdot)$. By \r{Pf} and the fact that $u^T=\phi$ on $\bo$, we get
\[
(u^T-\phi,\phi)_{H_{D}(\Omega)}=0.
\]
Then
\[
\|u^T-\phi\|^2_{H_{D}(\Omega)} = \|u^T\|^2_{H_{D}(\Omega)} - \|\phi\|^2_{H_{D}(\Omega)}\le \|u^T\|^2_{H_{D}(\Omega)}.
\]
Therefore, the energy  of the initial conditions in \r{2.3} satisfies the inequality
\be{2.4}
E_\Omega(w,T) = \|u^T-\phi\|^2_{H_{D}(\Omega)}  +\|u^T_t\|^2_{L^2(\Omega)}  \le E_\Omega(u,T).
\ee
Since the Dirichlet boundary condition is energy preserving, we get 
\[
E_{\Omega}(w,0) =    E_{\Omega}(w,T)\le  E_{\Omega}(u,T)\le E_{\R^n}(u,T)= E_{\Omega}(u,0) = \|\mathbf{f}\|^2_{\mathcal{H}}. 
\]
Therefore, 
\be{2.5}
\|\mathbf{K}\mathbf{f}\|^2_{H_{D}(\Omega)} = E_{\Omega}(w,0)\le \|\mathbf{f}\|^2_{H_{D}(\Omega)}.
\ee

We show next that actually the inequality above is strict, i.e., 
\be{2.6}
\|\mathbf{Kf}\|_{\mathcal{H}(\Omega)} < \|\mathbf{f}\|_{\mathcal{H}(\Omega)}, \quad \mathbf{f}\not=0.
\ee
Assume the opposite. Then for some $\mathbf{f}\not=0$, all inequalities leading to \r{2.5} are equalities. In particular, $ E_{\Omega}(w,T)=  E_{\R^n}(u,T)$. Then 
\[
u(T,x) = 0, \quad \mbox{for $x\not\in\Omega$}.
\]
By the finite domain of dependence then
\be{2.7}
u(t,x) = 0 \quad \mbox{when $\dist(x,\Omega)>|T-t|$}.
\ee
One the other hand, we also have 
\be{2.8}
u(t,x) = 0 \quad \mbox{when $\dist(x,\Omega)>|t|$}.
\ee
Therefore,
\be{2.8a}
u(t,x) = 0 \quad \mbox{when $\dist(x,\bo)>T/2, \;  -T/2\le t\le 3T/2$}.
\ee
In \cite{SU-thermo}, we used the fact that $u$ extends to an even function of $t$ that is still a solution of the wave equation because $f_2=0$ there. Then one gets that \r{2.8a} actually holds for $|t|<3T/2$. Then one concludes by Tataru's unique continuation theorem \cite{Tataru04} that $u=0$ on $[0,T]\times \Omega$, therefore, $f=0$.  We also noted there that the time interval $[0,T]$ is actually larger (twice as large) than what we need for the Neumann series to converge, see   \cite{SU-thermo_brain, S-U-InsideOut10}, where $T>T(\Omega)/2$ only. 

In the case under consideration, $f_2$ does not necessarily vanish. We modify the arguments as follows. From John's theorem (equivalent to Tataru's \cite[Theorem~2]{SU-thermo} in the Euclidean setting), we get that $u=0$ on $[0,T]\times \R^n\setminus \Omega$.  Then \cite[Theorem~2]{SU-thermo} implies that $u=0$ for $t=T/2$ and all $x$. By energy preservation, $\mathbf{f}=0$. 


Now, one has 
\be{2.9}
\|\mathbf{Kf}\|_{\mathcal{H}(\Omega)} \le \sqrt{\lambda_1} \|\mathbf{f}\|_{\mathcal{H}(\Omega)}, \quad \mathbf{f}\not=0,
\ee
where $\lambda_1$ is the largest eigenvalue of $\mathbf{K^*K}$. Then $\lambda_1<1$ by \r{2.6}.
\end{proof}

Denote by 
\[
\mathbf{B} := (\mathbf{\text Id}-\mathbf{K})^{-1}\mathbf{A}
\]
the left inverse of $\Lambda$ constructed in Theorem~\ref{thm2.1}. 

\begin{proposition}\label{pr_B} 
\[
\mathbf{B} : H^1_{(0)}([0,T]\times \bo) \to \mathcal{H}  \cong H_0^1(\Omega)\times L^2(\Omega)
\]
is a continuous map. 
\end{proposition}

\begin{proof}
Note first, that $\mathbf{A}$ has the mapping properties above, by the results in \cite{LasieckaLT}. Next, $(\mathbf{\text Id}-\mathbf{K})^{-1}$ is a bounded map in $\mathcal{H}$ by Theorem~\ref{thm2.1}. 
\end{proof}

Let $B_{1,2}$ be the components of $\mathbf{B}$ (that sends scalar functions to vector functions), i.e., $\mathbf{B}h= (B_1h, B_2h)$. Let $\Lambda_{1,2}$ be the components of $\Lambda$ (that sends vector functions to scalar functions, i.e., $\Lambda\mathbf{f} = \Lambda_1f_1+\Lambda_2f_2$. We can think of $\mathbf{B}$ as a $2$x$1$   matrix, and of $\Lambda$ as an $1$x$2$ matrix. Then $\mathbf{B}\Lambda [f_1,f_2]=[f_1,f_2]$, therefore,
\be{l5}
B_1\Lambda_1=\Id, \quad B_2\Lambda_2=\Id, \quad B_1\Lambda_2=0. \quad B_2\Lambda_1=0.
\ee
Set $\partial_t^{-1}g = \int_0^t g(s)\, \d s$ and let $\Delta_D$ be the Dirichlet realization of $\Delta$ in $\Omega$. 
We  have the following.
\begin{proposition}  \label{pr_12}
\be{16}
\begin{array}{rcll}
\Lambda_2c^2\Delta  &=& \partial_t \Lambda_1& \text{\rm on $ H_D(\Omega)\cap H^2(\Omega)  $},\\ 
 \partial_t\Lambda_2 &=& \Lambda_1, \quad \Lambda_2 = \partial_t^{-1} \Lambda_1 & \text{\rm on $ H_D(\Omega)  $},\\
\partial_t^{-1}\Lambda_2 &=& \Lambda_1 (c^2\Delta_D )^{-1} ,&
 \text{\rm on $ L^2(\Omega)  $},
\end{array}
\ee
\end{proposition}
\begin{proof}
Let $u$ solve \r{1'} with $\mathbf{f}=[f_1,0]\in D(\mathbf{P})$, see \r{s1}. 
Then $\partial_t u$ solves the wave equation, too, with Cauchy data $[0,c^2\Delta f_1]\in \mathcal{H}$. This proves the first relation in \r{16}. Similarly, let $u$ solve \r{1'} with $\mathbf{f}=[0,f_2]\in D(\mathbf{P})$. Then $\partial_t u$ solves the wave equation, too, with Cauchy data $[f_2,0]$. This proves the second relation.

For $f\in H_D(\Omega)$, we have $\partial_t\Lambda_2 f= \Lambda_1 f$ by what we just proved. Moreover, $\Lambda_2 f=0$ for $t=0$. This  proves the third relation in \r{16}.

Finally, let $f\in L^2(\Omega)$. Then $\partial_t \Lambda_1 (c^2\Delta_D )^{-1}f = \Lambda_2f$ by the first identity in \r{16}. Moreover, $\Lambda_1 (c^2\Delta_D )^{-1}f =0$ for $t=0$. This proves the fourth identity.
\end{proof}

Relations  \r{16} can be used to show that $\Lambda_2$ has a left inverse based on the result in \cite{SU-thermo} only. Indeed, set $B_2' := B_1\partial_t$. Then $B_2'\Lambda_2 = B_1\partial_t\Lambda_2 =B_1\Lambda_1=\Id$ on $H_D(\Omega)$. Analyzing the mapping properties of $B_2'$ and $\Lambda_2$ as in \cite{SU-thermo}, one gets that the composition  $B_1\Lambda_1$ is well defined and is a bounded operator on $L^2(\Omega)$; that therefore has to be identity. 


\section{Recovery of the speed $c$ when $f$ is known; The linearization $\delta\Lambda_1/\delta c^2$ w.r.t. the speed}
The non-linear problem of recovery of $c$ when $f$ is known, and its linearization were studied in detail by the authors in \cite{SU-Carleman}. We showed there that if $\delta c^2$ is a priori supported in a compact subset $\mathcal{K}$, then the linearization $\delta\Lambda_1/\delta c^2$ is Fredholm. We also gave conditions for uniqueness for both the linear and the non-linear problem. The geometric requirement is that there exists a foliation of $\bar\Omega$ by strictly convex surfaces. If $n=2$, non-trapping implies that condition. In all dimensions, if $c$ is close enough to a constant, that condition holds. In particular, if 
\be{I5'}
x\cdot\partial c<c\quad \text{in $\bar\Omega$},
\ee
then the parts   of the spheres $|x|=R$, $R>0$,  intersecting $\bar\Omega$ form a foliation of surfaces convex w.r.t.\ the metric $c^{-2}\d x^2$, see \cite{S-U-InsideOut10,SU-Carleman}.  There is an another condition as well: we require that $\Delta f(x)\not=0$, $\forall x\in K$. This condition fits well into our analysis; it is an if and only if condition for the operator $Q_N$ below to be elliptic.

\section{Analysis of the linearized operator $\delta\Lambda_1$} \label{sec_4} 
We derive a representation of the difference $\tilde \Lambda_1 \tilde f - \Lambda_1 f$ first. If one is interested in the linearization only, the computations are not really simpler --- one just has to drop the tilde in most of the formulas.  

Let $(c,f)$, $(\tilde c,\tilde f)$ be two pairs, and let $u$, $\tilde u$  be the corresponding solutions of \r{1}. Then
\begin{equation}   \label{w1}
\left\{
\begin{array}{rcll}
(\partial_t^2 -c^2\Delta)(\tilde u-u) &=&(\tilde c^2-c^2)\Delta \tilde u &  \mbox{in $(0,T)\times \R^n$},\\
(\tilde u-u)|_{t=0} &=&\tilde f- f,\\ \quad \partial_t (\tilde u-u)|_{t=0}& =&0.
\end{array}
\right.               
\end{equation}
Then 
\[
\tilde \Lambda_1 \tilde f - \Lambda_1 f = (\tilde u-u)\big|_{[0,T]\times \bo}.
\]
We have 
\be{w1a}
\tilde \Lambda_1 \tilde f - \Lambda_1 f = \Lambda_1(\tilde f-f)+ w\big|_{[0,T]\times \bo},
\ee
where $w$ solves
\begin{equation}   \label{w2}
\left\{
\begin{array}{rcll}
(\partial_t^2 -c^2\Delta)w &=&(\tilde c^2-c^2)\Delta\tilde  u &  \mbox{in $(0,T)\times \R^n$},\\
w|_{t=0} &=& 0,\\ \quad \partial_t w|_{t=0}& =&0.
\end{array}
\right.               
\end{equation}
Set  
\[
P:=c^2\Delta, \qquad 
h = c^{-2}( \tilde c^2 - c^2 ).
\]
 and let $P_D$ be the Dirichlet realization of $P$ in $\Omega$. Recall the notation  $\partial_t^{-1}g = \int_0^t g(s)\, \d s$. Denote by $R\mathbf{u}$ with  $\mathbf{u}=[u_1,u_2]$  the restriction of $u_1$ to $[0,T]\times \bo$. 
We have  
\[
w\big|_{[0,T]\times \bo} = R\int_0^t U(s)[0,hP  \tilde u(t-s,\cdot)]\, \d s.
\]
Take the $t$ derivative to get
\be{w3}
w_t\big|_{[0,T]\times \bo} =   \Lambda_2 hP  \tilde f+ W, \quad W:= R\int_0^t U(s)[0,h   P \tilde u_t(t-s,\cdot)]\, \d s.
\ee
Differentiate $W$ to get
\be{w3a}
\frac{\partial}{\partial t} W = R\int_0^t U(s)[0,h \partial_t^2 P  \tilde u (t-s,\cdot)]\, \d s ,
\ee
where we used the fact that $\partial_t  \tilde u=0$ for $t=0$, therefore 
 $\partial_t\Delta \tilde u=0$ for $t=0$. Differentiate one more time to get
\be{w3b}
\frac{\partial^2}{\partial t^2} W 
= \Lambda_2 hP^2 \tilde f   + 
R\int_0^t U(s)[0,h \partial_t^3 P\tilde u(t-s,\cdot)]\, \d s ,
\ee
Relations \r{w3} and \r{w3a} show that $W=O(t^2)$ at $t=0$. Therefore, 
\be{w3f}
W = \partial_t^{-2}\Lambda_2 h P^2\tilde f  + \partial_t^{-2} R\int_0^t U(s)[0,h \partial_t^3 P \tilde u(t-s,\cdot)]\, \d s := I_1+I_2 .
\ee
By Proposition~\ref{pr_12}, for the first term on the right we get
\[
I_1 = \Lambda_2P_D^{-1}\left(h P^2 \tilde f\right).
\]
The second term on the r.h.s.\ of \r{w3f} can be written in the form
\[
I_2 := \partial_t^{-2} R\int_0^t U(t-s)[0,h   P^2 \tilde u_t(s,\cdot)]\, \d s .
\]
We claim that $I_2=I_2'$, where
\[ 
I_2' = R\int_0^t U(t-s)[0,P_D^{-1}h  P^2 \tilde u_t(s,\cdot)]\, \d s .
\]
Indeed, $I_1'=0$ for $t=0$, and  direct differentiation shows that 
\[ 
\partial_t I_2' = R\int_0^t U(t-s)[P_D^{-1}h   P^2 \tilde u_t(s,\cdot),0]\, \d s .
\]
because $RP_D^{-1}=0$. Therefore, $\partial_t I_1'=0$ for $t=0$ as well. Differentiate one more time to get $I_2=I_2'$.  Therefore,
\[
W = \Lambda_2P_D^{-1} h  P^2 \tilde f + 
R\int_0^t U(s)[0,P_D^{-1}h   P^2 \tilde u_t(t-s,\cdot)]\, \d s .
\]
Compare with \r{w3} and repeat these arguments  to get for any $N=1,2,\dots$
\[
W = \sum_{k=1}^{N-1}\Lambda_2 P_D^{-k}h P^{k+1} \tilde f +  
R\int_0^t U(s)[0,P_D^{-N+1}h   P^{N}\tilde u_t(t-s,\cdot)]\, \d s .
\]
We want to emphasize that in all those formulas, $h$ is considered as an operator of multiplication, i.e., $\Lambda_2 P_D^{-k}h P^{k+1} \tilde f= \Lambda_2 (P_D^{-k}(h P^{k+1} \tilde f))$, etc.  
Combine this with \r{w1a}, \r{w3}, to get the following.
\begin{multline}  \label{w4_0}
\partial_t\left( \tilde \Lambda_1\tilde f - \Lambda_1 f\right)  \\= \Lambda_2 P (\tilde f-f) + 
  \sum_{k=0}^{N-1}\Lambda_2 P_D^{-k} h P^{k+1} \tilde f  +  
R\int_0^t U(s)[0,P_D^{-N+1} h    P^{N}\tilde u_t(t-s,\cdot)]\, \d s.
\end{multline}
Since $\tilde \Lambda_1\tilde f - \Lambda_1 f=0$ for $t=0$, integrating w.r.t. \ $t$, and applying Proposition~\ref{pr_12}, we get
\begin{multline}  \label{w4_1}
\tilde \Lambda_1\tilde f - \Lambda_1 f \\= \Lambda_1 (\tilde f-f) + 
  \sum_{k=1}^N\Lambda_1 P_D^{-k} h P^{k} \tilde f  +  \partial_t^{-1}
R\int_0^t U(s)[0,P_D^{-N}h    P^{N+1}\tilde u_t(t-s,\cdot)]\, \d s.
\end{multline}
Finally, using the arguments above, we write the last term on the right as
\[
R\int_0^t U(s)[P_D^{-N-1} h  P^{N+1}\tilde u_t(t-s,\cdot),0]\, \d s.
\]

We therefore proved the following. 
\begin{theorem}\label{thm_nonlin} Let $(c,f)$ and $(\tilde c, \tilde f)$ be two  pairs in $C^{2N}(\bar\Omega)\times H^{2N+2}(\Omega)\cap H_0^1(\Omega)$ with $c>0$, $\tilde c>0$ in $\bar\Omega$. 
 Then 
\begin{align}  \nonumber
 \tilde \Lambda_1\tilde f - \Lambda_1 f   = &\Lambda_1 (\tilde f-f) + 
\Lambda_1  \sum_{k=1}^N\big(c^{2}\Delta_D\big)^{-k}\left(c^{-2}(\tilde c^2 - c^2) \big(c^2\Delta\big)^{k} \tilde f\right) \\ &  + 
R\int_0^t U(s)\left[\big(c^{2}\Delta_D\big)^{-N-1}\left( c^{-2}(\tilde c^2 - c^2)  \big(c^{2}\Delta\big)^{N+1}\tilde u_t(t-s,\cdot)\right),0\right]\, \d s.\label{w4}
\end{align}
\end{theorem}

We define the linearization  $\delta \Lambda_1\{\delta f,\delta c^2\}$ at $(c,f)$ as the derivative at $\eps=0$ of $\Lambda_1$ with speed  $c_\eps^2=c+\eps \delta c$ and source $f_\eps=f+\eps\delta f$. 

\begin{corollary} For any $N=1,2,\dots$, the linearized operator $\delta \Lambda_1\{\delta f,\delta c^2\}$ at $(c,f)$ has the form
\be{w5}
 \delta \Lambda_1\{\delta f,\delta c^2\} = \Lambda_1\!\left( \delta f+Q_N  c^{-2}\delta c^2\right) +R_N c^{-2}\delta c^2,
\ee
where
\be{w7}
Q_N h = \sum_{k=1}^N  \big(c^{2}\Delta_D\big)^{-k}\left( h \big(c^2\Delta\big)^{k}  f\right) ,
\ee
and 
\be{w7a}
R_N h =  R\int_0^t U(s)\left[\big(c^{2}\Delta_D\big)^{-N-1}\left( h  \big(c^{2}\Delta\big)^{N+1} u_t(t-s,\cdot)\right),0\right]\, \d s .
\ee
\end{corollary}

\begin{proof}[Proof of Theorem~\ref{thm_lin}]
Note that for $(c,f)\in C^\infty$, the operator $R_N$ is smoothing $2N+2$ degrees, i.e., 
\be{w7b}
R_N: H^s_0(\mathcal{K})\to H^{s+2N+2}([0,T]\times\bo),
\ee
 and $Q_N$ is a \PDO\ (in the interior of $\Omega$) with principal symbol $\Delta f(x)|\xi|^{-2}$. So for $\delta c^2$ supported in a compact set in $\Omega$, we can write
\be{w6}
 \delta \Lambda_1\{\delta f,\delta c^2\} = \Lambda_1\!\left( \delta f+\left(c^{-2}\Delta f |D|^{-2}  +\text{l.o.t.} \right)\delta c^2\right) +R_\infty\delta c^2,
\ee
where $R_\infty$ is smoothing, and $\text{l.o.t.}$ stands for ``lower order terms'', i.e., for a pseudo-differential operator of order $-3$, which we can always assume to have a proper support. 

This result is not surprising. All singularities of the kernel of $\delta\Lambda_1$ are of conormal type, at $\dist(x,y)=t$, where $\dist$ is the distance in the metric $c^{-2}\d x^2$. This can be seen from the representation 
\be{l}
\delta\Lambda_1\{\delta c^2,\delta f\} = \Lambda_1\delta f+R\int_0^t U(s)[0,\Delta u(t-s,\cdot)\delta c^2]\,\d s.
\ee
 The operator $Q_N$ can be explained by those singularities. Next, $R_N$ depends on the smooth part of the kernel, and in particular, the behavior of the kernel inside that geodesic ball.  

An important observation is that if the expression in the parentheses in \r{w5} vanishes (and this is an explicit condition on $\delta f$ and $\delta c^2$), then the linearization is very smooth. 
If for a moment we ignore $R_\infty$, we get a kernel of infinite dimension. Indeed, given $\delta c^2$, compactly supported in $\Omega$, we can always find $\delta f$ so that  $ \delta f+\left( c^{-2}\Delta f |D|^{-2} +\text{l.o.t.} \right)\delta c^2=0$. If we are given a compactly supported (in $\Omega$) function $\delta f$, and if $\Delta f\not=0$ on $\supp \delta f$,   we can find $\delta c^2$ at least away from a finitely dimensional space, so that the expression above vanishes as well. The effect of $R_\infty$ will be the following. Even though we still do not know whether we can determine both $\delta f$ and $\delta c^2$, we can claim that even of we could, the problem would be  unstable. 

More precisely, given $N>1$, let $V_N$ be the linear space
\be{V1}
V_N= \left\{ (\delta f, h)\in H_D(\Omega)\times L^2(\mathcal{K}) ;\;  \delta f+ Q_Nh=0\right\}. 
\ee
Then for the linearization $ \delta \Lambda_1\{\delta f,\delta c^2\}$ we have
\be{smooth}
  V_N\ni (\delta f,c^{-2}\delta c^2) \longmapsto  \delta \Lambda_1  \{\delta f,\delta c^2\} \in  H^{2N+2}_{(0)}([0,T]\times\bo).
\ee
This can easily be generalized to $(\delta h,f)$ belonging to negative Sobolev spaces.  
We then complete the proof of Theorem~\ref{thm_lin} by a well known argument, see, e.g., \cite{SU-JFA09} or the remark below. 
\end{proof}

\begin{remark}\label{rem1}
It is easy to see also that $\delta \Lambda_1$ does not  satisfy the following conditional stability estimate either:
\be{H}
\|\delta f\|_{H^{s_1}(\Omega)} + \|\delta c^2\|_{H^{s_1}(\mathcal{K})}\le C\big\|\delta\Lambda_1  \{\delta f,\delta c^2\} \big\|_{H^{s_2}}^{\mu},\quad 0<\mu<1,
\ee
under the condition 
\be{H1}
\|\delta f\|_{H^{s_3}(\Omega)} + \|\delta c^2\|_{H^{s_3}(\mathcal{K})}\le A
\ee
regardless of how we choose $s_1$, $s_2$, $s_3$ and $A>0$. Indeed, fix $\phi\in C_0^\infty$, supported in the interior of $\mathcal{K}$. Set
\[
h_\lambda = \lambda^{-s_3} e^{\i \lambda \omega\cdot x}\phi, 
\]
where $\lambda>0$, $s_3>0$ and $\omega\in S^{n-1}$ are fixed. Let $\delta f_\lambda$ be such that $(\delta f_\lambda,h_\lambda)\in V_N$, $N\gg1$. Then \r{H1} is satisfied. On the other hand, with $c_\lambda^2 = (1+h_\lambda)c^2$,
\[
\big\|\delta\Lambda_1  \{\delta f_\lambda,\delta c^2_\lambda\} \big\|_{H^{s_2}}= 
\|R_Nh_\lambda\|_{H^{s_2}}\le C_N\lambda^{s_2-s_3-2N-2},
\]
see \r{w7b}, while
\[
\|\delta c_\lambda^2\|_{H^{s_1}}\ge \lambda^{s_1-s_3}/C. 
\]
Therefore we get a contradiction with \r{H} for  $N\gg1$.
\end{remark}


\end{document}